\documentclass[11pt]{article}
\usepackage{times}
\usepackage{amsmath,amsfonts,amstext,latexsym,amssymb,amsbsy,amsopn,amsthm,eucal,slashed,amstext,enumerate,accents}
\usepackage{mathrsfs}
\usepackage{cases}
\usepackage{color}
\usepackage{txfonts}
\usepackage{dsfont}
\usepackage{graphicx}  
\usepackage{comment}

\allowdisplaybreaks[4]
\usepackage{scalerel,stackengine}
\stackMath
\newcommand\reallywidehat[1]{%
\savestack{\tmpbox}{\stretchto{%
  \scaleto{%
    \scalerel*[\widthof{\ensuremath{#1}}]{\kern-.6pt\bigwedge\kern-.6pt}%
    {\rule[-\textheight/2]{1ex}{\textheight}}
  }{\textheight}%
}{0.5ex}}%
\stackon[1pt]{#1}{\tmpbox}%
}
\usepackage[colorlinks,
            linkcolor=blue,      
            anchorcolor=blue,  
            citecolor=blue,        
            ]{hyperref}

\allowdisplaybreaks[4]

\begin{document}

\setlength{\abovedisplayshortskip}{3mm}
\setlength{\belowdisplayshortskip}{3mm}
\setlength{\abovedisplayskip}{3mm}
\setlength{\belowdisplayskip}{3mm}

\newtheorem{theorem}{Theorem}[section]
\newtheorem{corollary}[theorem]{Corollary}  
\newtheorem{lemma}[theorem]{Lemma}
\newtheorem{definition}[theorem]{Definition}
\newtheorem{question}[theorem]{Question}
\newtheorem{proposition}[theorem]{Proposition}
\newtheorem{remark}[theorem]{Remark}
\newtheorem{example}[theorem]{Example}
\newtheorem{theoremalph}{Theorem}
 \renewcommand\thetheoremalph{\Alph{theoremalph}}

\newenvironment{sequation}{\begin{equation}\small}{\end{equation}}
\newenvironment{tequation}{\begin{equation}\tiny}{\end{equation}}
\renewcommand{\thefootnote}{\fnsymbol{footnote}}


\title
{\bf\Large $W$-entropy formulas and Langevin deformation on the $L^q$-Wasserstein space over Riemannian manifolds}

\author
{\ \ Rong Lei\thanks{ Research of R. Lei has been supported by National Key R\&D Program of China (No. 2020YF0712700) and NSFC No. 12171458. },\quad 
Xiang-Dong Li\footnote{Research of X.-D. Li has been supported by National Key R\&D Program of China (No. 2020YF0712700), NSFC No. 12171458, and Key Laboratory RCSDS, CAS, No. 2008DP173182.},\quad 
Yu-Zhao Wang \footnote{Research of Y.-Z. Wang has been supported by Fundamental Research Program of Shanxi Province(No. 202303021211001).} 
}

\date{}

\maketitle
\numberwithin{theorem}{section}
\setcounter{tocdepth}{2}
\setcounter{secnumdepth}{2}

\begin{abstract}
We first prove the $W$-entropy formula and rigidity theorem for the geodesic flow on the $L^q$-Wasserstein space over a complete Riemannian manifold with  bounded geometry condition. Then we introduce the Langevin deformation on the $L^q$-Wasserstein space over a complete Riemannian manifold,
which interpolates between the $p$-Laplacian heat equation and the geodesic flow on the $L^q$-Wasserstein space, where ${1\over p}+{1\over q}=1$, $1< p, q<\infty$. The local existence, uniqueness and regularity of the Langevin deformation on the $L^q$-Wasserstein space over the Euclidean space and a compact Riemannian manifold are proved for $q\in [2, \infty)$.   We further prove the $W$-entropy-information  formula and the rigidity theorem for the Langevin deformation on the $L^q$-Wasserstein space over an $n$-dimensional complete Riemannian manifold with non-negative Ricci curvature, where $q\in (1,\infty)$.

\vspace{2mm}

\vspace{2mm}
\end{abstract}


\section{Introduction}

In recent years, the optimal transport theory has been an important topic in the interplay among analysis,
PDEs, differential geometry, and probability theory \cite{AGS, JKO, BB, McC, Otto01, OV, V1, V2}. A cornerstone of the
Lott–Villani–Sturm theory \cite{LV, Sturm, RensSturm, EKS, V1, V2} for the synthetic geometry of Ricci curvature on metric measure
spaces is the convexity of the Boltzmann entropy and the R\'enyi entropy along
Wasserstein geodesics. 

Let $(M, g)$ be a complete Riemannian manifold, $dv(x)=\sqrt{{\rm det}g(x)}dx$ the standard 
volume measure, and $p>1$. 
 Let $P_p(M)$ (resp. $P_p^\infty(M)$) be the $L^p$-Wasserstein space (resp. the smooth $L^p$-Wasserstein space) of all probability measures $\rho\hspace{0.2mm} dv$ with density function (resp. with smooth density function) $\rho$ on $M$ such that $\int_M d^p(o, x)\rho(x)\hspace{0.2mm} dv(x)<\infty$, where $d(o, \cdot)$ 
 denotes the distance function from a fixed point  $o\in M$.
 
Let $\mu_0, \mu_1\in P_p(M)$.  In 1940s,  Kantorovich \cite{Kant}, Kantorovich and Rubinstein \cite{KR}  introduced the $L^p$-Wasserstein
distance between $\mu_0$ and $\mu_1$ as follows 
$$
W_p^p\left(\mu_0, \mu_1\right):=\inf_{\pi\in \Pi} \int_{M\times M} d(x, y)^p \mathrm{~d} \pi(x, y),
$$
where $\Pi$ is the set of coupling measures $\pi$ of $\mu_0$ and $ \mu_1$ on $
M\times M$, i.e., $\Pi=\{\pi\in P(M\times M), \pi(\cdot, M)=\mu_0, \pi(M, \cdot)=\mu_1\}$, where $P(M\times M)$ is the set of probability measures on $M\times M$. Moreover, Kantorovich \cite{Kant} (see also \cite{KR, V1}) proved that 
\begin{equation*}\label{L2WD}
\frac1 pW^p_p(\mu_0,\mu_1)=\sup_{\phi\in C_b(M)}\left\{\int\phi d\mu_0+\int\phi^cd\mu_1\right\},
\end{equation*}
where $\phi^c$ is the conjugate of $\phi$, defined by 
$$
\phi^c(x):=\inf_{y\in M}\left[\frac{d(x, y)^p}{p}-\phi(y)\right].
$$

Suppose that  $M=\mathbb{R}^n$, $\mu_0 = \rho_0dx$ and $\mu_1 =\rho_1dx$. In  \cite{BB}, Benamou and Brenier showed that the $L^2$-Wasserstein distance has a natural hydrodynamical interpretation. More precisely, we have the following Benamou-Brenier variational formula
\begin{equation}\label{BBF1}
W^2_2(\mu_0,\mu_1)=\inf\left\{\int^1_0\int_{\mathbb R^n}|{\bf v}(x,t)|^2\rho(x,t)\ dxdt:\partial_t\rho+\nabla\cdot(\rho {\bf v})=0,\; \rho(0)=\rho_0,\;\rho(1)=\rho_1\right\}.
\end{equation}
Moreover, the infimum of the right-hand side in \eqref{BBF1} is achieved by $\rho$ and ${\bf v}=\nabla \phi$ which satisfy the continuity equation and the Hamilton-Jacobi equation
\begin{equation} 
\left\{\begin{aligned}
    &\frac{\partial\rho}{\partial t}+\nabla\cdot\left(\rho\nabla \phi\right)=0, \label{2NDE0}\\
   &\frac{\partial \phi}{\partial t}+\frac12|\nabla \phi|^2=0.
\end{aligned}
\right.
\end{equation}
In view of this, we can regard any solution  $(\rho, \phi)$ of the above equations as a geodesic flow on the tangent bundle $TP_2(\mathbb{R}^n)$ over the $L^2$-Wasserstein space $P_2(\mathbb{R}^n)$. 

In the optimal transport theory, an infinite dimensional Riemannian structure has been introduced by Otto \cite{Otto01} on the $L^2$-Wasserstein space over Euclidean spaces. See also \cite{OV, LV, Lott, LiLi3, LiLi4} for its extension on Riemannian manifolds. More precisely, 
the tangent space of $P_2(M)$ at $\rho dv$ is defined as 
$$
T_{\rho dv}P_2(M)=\left\{~s=-\nabla\cdot (\rho \nabla \phi):\ \phi\in W^{1, 2}(M, \rho dv), \ \ \|s\|^2_2:=\int_M |\nabla\phi|^2\rho dv<+\infty\ \right\},$$ 
where $W^{1, 2}(M, \rho dv)=\left\{~\phi\in L^2(M, \rho dv): \int_M |\nabla \phi|^2\rho dv<\infty\right\}$. When we restrict to $P_2^\infty(M)$, we denote the tangent space of $P_2^\infty(M)$ at $\rho dv$ by $T_{\rho dv}P_2^\infty(M)$. 
For any $s_i=-\nabla\cdot (\rho \nabla \phi_i)\in T_{\rho dv}P_2(M)$, $i=1, 2$, the inner product on $T_{\rho dv}P_2(M)$ is defined by 
$$
\langle\langle s_1, s_2\rangle\rangle =\int_M \langle \nabla \phi_1, \nabla \phi_2\rangle\rho dv.$$

In his seminal work \cite{Perelman}, Perelman introduced the $W$-entropy and proved its monotonicity 
for  Ricci flow. This plays a crucial role in the final solution of the Poincar\'e conjecture. Inspired by Perelman \cite{Perelman}  and related works \cite{Ni, LiXD2}, S. Li and the second named author of this paper \cite{LiLi4, LiLi3} introduced the $W$-entropy and proved its monotonicity for the geodesic flow  on the $L^2$-Wasserstein space over a Riemannian manifold with non-negative Ricci curvature. More precisely, we have   
 \begin{theorem}\label{LLinfty}
 Let $(M, g)$  be a complete Riemannian manifold with bounded geometry condition\footnote{We say that $(M,g)$ satisfies the bounded geometry condition if  the Riemannian curvature tensor $\mathrm{Riem}$ and its covariant derivatives $\nabla^k \mathrm{Riem}$ are uniformly bounded on $M$ for $k = 1, 2, 3$.}. Let 
  $(\rho(t), \phi(t), t\in [0, T])$ be  a smooth geodesic flow in $TP_2^\infty(M)$. Let
 $$H_n(\rho(t)):={\rm Ent}(\rho(t))+{n\over 2}\left(1+\log(4\pi t^2)\right),$$ 
 where ${\rm Ent}(\rho(t)):=\int_M \rho(t)\log \rho(t)dv$ is the Boltzmann entropy. Define the $W$-entropy for the geodesic flow by 
 \begin{eqnarray*}
 W_n(\rho(t)):={d\over dt}(tH_n(\rho(t))).
 \end{eqnarray*}
 Then for all $t>0$, we have
 \begin{eqnarray}
 {1\over t}{d\over dt}W_n(\rho(t))&=&\int_M \left[\left|\nabla^2\phi-{g\over t}\right|^2+\mathrm{Ric}(\nabla \phi, \nabla \phi) \right]\rho\hspace{0.2mm} dv.\label{Wgeo}
 \end{eqnarray}
 In particular, if $\mathrm{Ric}\geq 0$, then ${d\over dt}W_n(\rho(t))\geq 0$.  Moreover, under the condition that $\mathrm{Ric}\geq 0$, ${d\over dt}W_n(\rho(t))=0$ holds at some $t=t_0>0$ if and only if $(M, g)$ is isometric to $\mathbb{R}^n$, and $(\rho, \phi)=(\rho_n, \phi_n)$, where for  $t>0$, $x\in \mathbb{R}^n$, 
 \begin{eqnarray}
 \rho_n(t, x)={1\over (4\pi t^2)^{n/2}}e^{-{\|x\|^2\over 4t^2}},\ \ 
 \phi_n(t, x)={\|x\|^2\over 2t},\label{phiminfty}
 \end{eqnarray}
  is a special solution to the geodesic flow on $T\hspace{0.2mm}P^\infty_2(\mathbb{R}^n)$. 
 \end{theorem}
 
It is natural to ask the question whether we can extend the above result to the $L^q$-Wasserstein space on $\mathbb{R}^n$ or a complete Riemannian manifold, where $q>1$. Let $p=\frac q{q-1}$. 
Following \cite{Otto01, AGS, LiLi3, LiLi4}, 
the tangent space of $P_q(M)$ at $\rho dv$ is defined as 
$$
T_{\rho dv}P_q(M)=\left\{~s=-\nabla\cdot (\rho|\nabla\phi|^{p-2}\nabla \phi):\ \phi\in W^{1, p}(M, \rho dv), \ \ \|s\|_q^q:=\int_M |\nabla\phi|^p\rho dv<+\infty\ \right\},$$ 
where $W^{1, p}(M, \rho dv)=\left\{~\phi\in L^p(M, \rho dv): \int_M |\nabla \phi|^p\rho dv<\infty\right\}$. When we restrict to $P_q^\infty(M)$, we denote the tangent space of $P_q^\infty(M)$ at $\rho dv$ by $T_{\rho dv}P_q^\infty(M)$. 
In view of this, the gradient flow of a smooth functional $\mathcal{V}$ on $P_q(M)$ is given by 
\begin{equation}\label{pgradflow1}
\partial_t\rho+\nabla\cdot\left(\rho \left|\nabla \phi\right|^{p-2}\nabla \phi\right)=0, 
\end{equation}
where $\phi={\delta \mathcal V\over \delta\rho}$ denotes the $L^2$-derivative of $\mathcal V$ with respect to $\rho$. In particular, if $\mathcal V(\rho)=\int_M V(\rho)dv$ with $V\in C^1(\mathbb{R})$, then the corresponding gradient flow reads 
 \begin{equation*}\label{pgradflow2}
\partial_t\rho+\nabla\cdot\left(\rho|\nabla V'(\rho)|^{p-2}\nabla V'(\rho)\right)=0.
\end{equation*}

Similarly to Benamou and Brenier \cite{BB} for $q=2$, Brasco \cite{Brasco} proved the following variational formula
\begin{equation}\label{Wq}
W_q(\mu_0,\mu_1)=\inf\left\{\int^1_0\int_M|\textbf{v}(x,t)|^q\rho(x,t)dvdt:\partial_t\rho+\nabla\cdot(\rho\textbf{v})=0,\; \rho(0)=\rho_0,\;\rho(1)=\rho_1\right\}^{1\over q},
\end{equation}
Moreover, the infimum of the right hand side in \eqref{Wq} is achieved by $\rho$ and $\textbf{v}=|\nabla \phi|^{p-2}\nabla \phi$ which satisfy the following $p$-continuity equation and the $p$-Hamilton-Jacobi equation
\begin{equation} \label{NDE0}
\left\{\begin{aligned}
    &
    \frac{\partial}{\partial t}\rho+\nabla\cdot\left(\rho|\nabla \phi|^{p-2}\nabla \phi\right)=0, \\
  &\frac{\partial}{\partial t}\phi+\frac1p|\nabla \phi|^p=0.
\end{aligned}
\right.
\end{equation} 
In view of this, we can regard any solution $(\rho, \phi)$ to the above equations as a geodesic flow on the tangent bundle $TP_q(M)$ over the $L^q$-Wasserstein space $P_q(M)$ for any $q>1$. 

We now state the first main theorem of this paper, which shows the $W$-entropy formula for the geodesic flow on the $L^q$-Wasserstein space over complete Riemannian manifolds with bounded geometry condition.

\begin{theorem}\label{CEntForm}
Let $p>1$, $q=\frac p{p-1}$. Let $M$ be a complete Riemannian manifold with
bounded geometry condition and $(\rho,\phi)$ a smooth solution to the $L^q$-geodesic equation \eqref{NDE0} with suitable growth condition\footnote{For the exact description of the suitable growth condition, see Proposition \ref{EDF}.}.
Assume that $\int_M \rho(0, x)dv(x)=1$.   Define the relative entropy
\begin{equation*}\label{ReEnt}
{\rm Ent}_{n,p}(\rho,t):=\int_M\rho\log\rho\, dv+\frac nq\left(1+\log\left(c_{n,p}^{-\frac qn}t^q\right)\right),
\end{equation*}
where 
\begin{equation}\label{cpn}
c_{n,p}=(pq^{p-1})^{-\frac np}\pi^{-\frac n2}\frac{\Gamma(\frac n2+1)}{\Gamma(\frac nq+1)}.\end{equation} 
Define the $W$-entropy as follows
\begin{align}\label{WEntropy}
W_{n,p}(\rho,\phi,t):=\frac d{dt}(t{\rm Ent}_{n,p}(\rho,t)).
\end{align}
Then 
\begin{align}\label{NLLWEnt}
\frac{d}{dt}W_{n,p}(\rho,\phi,t)=&t\int_M\left(\left||\nabla \phi|^{p-2}\nabla_i\nabla_j \phi-\frac{a_{ij}}{t}\right|_A^2+|\nabla \phi|^{2p-4}{\rm Ric}(\nabla \phi,\nabla \phi)\right)\rho\,  dv, 
\end{align}
where  $A=(A^{ij})$ is defined by
\begin{equation}\label{Atensor}
A:= {\rm{g}}+(p-2)\frac{\nabla \phi\otimes \nabla \phi}{|\nabla \phi|^2},
\end{equation}
and $a=(a_{ij})$ is the inverse of $(A^{ij})$, and for a second order tensor $T$, $|T|^2_A=\sum_{i,j,k,l}A^{ik}A^{jl}T_{ij}T_{kl}$.

In particular, if $\mathrm{Ric}\geq 0$, then $\frac{d}{dt}W_{n,p}\geq 0$. Moreover, under the condition that $\mathrm{Ric}\geq 0$, 
$\frac{d}{dt}W_{n,p}(\rho,\phi)=0 $ holds at some $t = t_0 > 0$
%
if and only if $(M, g)$ is isometric to the Euclidean space $\mathbb{R}^n$
and $(\rho,\phi)=(\rho_{n, p},\phi_{n, p})$, where for $n\in\mathbb{N}$, $p>1$, $t>0$, $x\in\mathbb{R}^n$, $(\rho_{n, p},\phi_{n, p})$ is a  special solution to  the $L^q$-geodesic flow \eqref{NDE0} on the Euclidean space $\mathbb{R}^n$. More precisely, 
\begin{equation}\label{special}
\rho_{n, p}(x,t)=c_{n,p}t^{-n}
\exp\left\{-(p-1)\frac{\|x\|^q}{(pt)^q}\right\}, \  \ \ \ \phi_{n, p}(x,t)=\frac{\|x\|^q}{qt^{q-1}},
\end{equation} 

\end{theorem}

The second main result of this paper is  the $W$-entropy-information formula for the Langevin deformation of flows on the $L^q$-Wasserstein space over complete Riemannian manifolds. See Theorem \ref{Ent-Riccati}.  To save the length of the introduction, we will introduce the Langevin deformation and state this result in Section 2. 

The rest of this paper is organized as follows. In Section 3, we derive some variational formulas for the geodesic flow on the $L^q$-Wasserstein space. In Section 4,  we prove the first main result of this paper, i.e., Theorem \ref{CEntForm}. In Section 5, we prove the local existence and uniqueness to the Cauchy problem for the compressible $p$-Euler equation with damping and 
the Langevin deformation of flows on the $L^q$-Wasserstein space over the Euclidean space and a compact Riemannian manifold for $q\in [2, \infty)$. 
In Section 6, we first prove two variational formulas for the Hamiltonian and Lagrangian on the $L^q$-Wasserstein space, then we prove the second main result of this paper, i.e., Theorem \ref{Ent-Riccati}. 

We would like to point out that the main results in this paper can be naturally extended  to complete Riemannian manifolds with weighted volume measure satisfying the so-called CD$(0, m)$-curvature-dimension condition. To save the length of this paper, we omit them. In the case $p=q=2$, see \cite{LiLi3, LiLi4}. 

\section{Langevin deformation of flows}

In \cite{LiLi4, LiLi3}, S. Li and the second named author of this paper  introduced the Langevin deformation of flows on $TP_2(M)$ as smooth solutions to the following equations
\begin{equation} \label{LLflow}
\left\{\begin{aligned}
   &\partial_t \rho=-\nabla\cdot(\rho \nabla \phi),\\
 &  c^2\left({\partial\phi\over \partial t}+{1\over 2}|\nabla \phi|^2\right)=- \phi-\nabla {\delta \mathcal{V} \over \delta \rho},
  \end{aligned}
\right.
\end{equation}
 where $c\in (0, \infty)$. Heuristically, when $c\rightarrow 0$, we have the gradient flow of $\mathcal{V}$ on $P_2(M)$
 \begin{eqnarray*}
 \partial_t \rho=-\nabla\cdot\left(\rho \nabla {\delta \mathcal{V} \over \delta \rho}\right),
\end{eqnarray*}
 and when $c\rightarrow\infty$, we have the geodesic  flow \eqref{2NDE0} on $TP_2(M)$.  In the case  $M$ is $\mathbb{R}^n$ or a compact Riemannian manifold and $\mathcal{V}(\rho)={\rm Ent}(\rho):=\int_M \rho\log \rho dv$ (i.e., the Boltzmann entropy) or $\mathcal{V}(\rho)={\rm Ent}_\gamma(\rho):={1\over \gamma-1}\int_M \rho^\gamma dv$ (i.e., the R\'enyi entropy) with $\gamma\neq 1$, the local existence and uniqueness to the Cauchy problem for the 
 Langevin deformation \eqref{LLflow} have been proved. If the initial data is small in the sense of Sobolev norm, then the global existence and uniqueness result also hold. In the case $\mathcal{V}(\rho)={\rm Ent}(\rho)=\int_M \rho\log \rho dv$, they proved that the convergence of the Langevin deformation when $c\rightarrow 0$ and $c\rightarrow \infty$ respectively. Moreover, the $W$-entropy-information formula was also proved for the Langevin deformation on the $L^2$-Wasserstein space. For details, see \cite{LiLi3, LiLi4}.

We now introduce the Langevin deformation of flows on $TP_q(M)$ as smooth solutions to the following equations
 \begin{equation} \label{NDEC1}
\left\{\begin{aligned}
    &\frac{\partial\rho}{\partial t}+\nabla\cdot\left(\rho|\nabla \phi|^{p-2}\nabla \phi\right)=0, \\
   &c^p\left(\frac{\partial \phi}{\partial t}+\frac1p|\nabla \phi|^p\right)=-\phi-{\delta \mathcal{V}\over \delta \rho}.
\end{aligned}
\right.
\end{equation}
where $c\in (0, \infty)$. Similarly to \cite{LiLi3, LiLi4}, the Langevin deformation \eqref{NDEC1} interpolates between the $p$-Laplacian heat equation and the geodesic flow on the $L^q$-Wasserstein space. Heuristically, when $c\rightarrow\infty$, we have the geodesic  flow \eqref{NDE0} on $TP_q(M)$. When $c\rightarrow 0$, we have the gradient flow \eqref{pgradflow1} of $\mathcal{V}$ on $P_q(M)$.

The Langevin deformation \eqref{NDEC1} has a close connection with hydrodynamical equations. Indeed, let  ${\bf u}=\nabla \phi$ and ${\bf v}=|\nabla \phi|^{p-2}\nabla \phi$, where $p>1$, then the Langevin deformation $\eqref{NDEC1}$ reads
\begin{equation} \label{NDEC3}
\left\{\begin{aligned}
    &\frac{\partial\rho}{\partial t}+\nabla
    \cdot\left(\rho {\bf v}\right)=0, \\
   &c^p\left(\frac{\partial {\bf u}}{\partial t}+\nabla_{\bf v} {\bf u}\right)=-\bf u-\nabla\frac{\delta\mathcal{V}}{\delta\rho},
\end{aligned}
\right.
\end{equation}
which can be viewed as the  compressible $p$-Euler equation with damping on $M$. 

In the following, we consider $\mathcal{V}(\rho)=\mathrm{Ent}(\rho)=\int_M \rho\log \rho dv$. Then the $L^p$-Langevin deformation of flows  \eqref{NDEC1} reads
\begin{equation} \label{Lan}
\left\{\begin{aligned}
    &\frac{\partial\rho}{\partial t}+\nabla\cdot\left(\rho|\nabla \phi|^{p-2}\nabla \phi\right)=0, \\
   &c^p\left(\frac{\partial \phi}{\partial t}+\frac1p|\nabla \phi|^p\right)=-\phi-\log\rho-1.
\end{aligned}
\right.
\end{equation}
The local existence and uniqueness of the smooth solution to the Cauchy problem for the $L^p$-Langevin deformation of flows  \eqref{Lan} will be proved in Section \ref{LDF} for $c\in(0,+\infty)$ and $p\geq 2$. 

The second main result of this paper is the following $W$-entropy-information formula. When $p=q=2$, it was due to \cite{LiLi3, LiLi4}.

\begin{theorem}\label{Ent-Riccati} 
Let $c>0$ and $M$  an $n$-dimensional complete Riemannian manifold with bounded geometry condition.
Let $(\rho(t),\phi(t))$ be a smooth solution to the Langevin deformation  \eqref{Lan} with reasonable growth condition on $P_q(M)$ (see  Proposition  \ref{EDF} below).  Then 
\begin{align}\label{dtH4}
&\frac{d^2}{dt^2}{\rm Ent}(\rho)+\frac{p-1}{c^p}\frac{d}{dt}{\rm Ent}(\rho)+\frac1{c^p}\int_M|\nabla\phi|^{p-2}|\nabla\log\rho|^2_A\rho\,dv=\int_M|\nabla \phi|^{2p-4}\left(|\nabla^2 \phi|^2_A+{\rm Ric}(\nabla \phi,\nabla \phi)\right)\rho\,dv.
\end{align}
Define the 
relative Boltzmann entropy by
$$
{\rm Ent}_{c,n,p}(\rho(t)):={\rm Ent}(\rho(t))+\frac nq\left(1+\log\left(c_{n,p}^{-\frac qn}w^q(t)\right)\right),
$$
where $c_{n, p}$ is a constant given by \eqref{cpn}, and $w: (0,T]\to\mathbb{R}$ is a smooth solution to the following equation
\begin{equation}\label{pode}
c^p\ddot{w}(t)+(p-1)\dot{w}(t)=\frac{p-1}{p^{q-1}}\frac{\dot{w}^{2-q}(t)}{w(t)}.
\end{equation}
Let $\alpha(t)=\frac{\dot{w}(t)}{w(t)}$.
Define the $W$-entropy for the Langevin deformation \eqref{Lan} by
\begin{equation}\label{Wcn}
{W}_{c,n,p}(\rho(t),t):={\rm Ent}_{c,n,p}(\rho(t))+\eta(t)\frac d{dt}{\rm Ent}_{c,n,p}(\rho(t)),
\end{equation}
where 
 $$\eta(t):=-w^2(t)e^{\frac{(p-1)t}{c^p}}\int^tw^{-2}(s)e^{-\frac{(p-1)s}{c^p}}ds$$ 
 is a solution to 
\begin{equation}\label{eta}
\frac{1+\dot{\eta}(t)}{\eta(t)}=2\alpha(t)+\frac{p-1}{c^p}.
\end{equation}
Define the relative Fisher information by
\begin{equation}\label{Icn}
I_{c,n,p}(\rho(t),\phi(t)):=\int_M|\phi(t)|^{p-2}|\nabla\log\rho(t)|_A^2\rho(t)\,dv-\frac{p-1}{p^{q-1}}\frac{n\alpha^{2-q}(t)}{w^{q}(t)}.
\end{equation}
Then the following $W$-entropy-information formula holds
\begin{align}\label{wcnentropy}
\frac1{\eta(t)}\frac{d}{dt}{W}_{c,n,p}(\rho(t),t)+\frac1{c^p}I_{c,n,p}(\rho(t),\phi(t))
=\int_M\left[\Big||\nabla \phi|^{p-2}\nabla_i\nabla_j\phi-\alpha(t)a_{ij} \Big|^2_A+|\nabla \phi|^{2p-4}{\rm Ric}(\nabla \phi,\nabla \phi)\right]\rho\,dv.
\end{align}
In particular, if $\mathrm{Ric}\geq 0$, then for all $t>0$, the $W$-entropy-information inequality holds
\begin{equation}\label{WEI} 
\frac1{\eta(t)}\frac{d}{dt}{W}_{c,n,p}(\rho(t),t)+\frac1{c^p}I_{c,n,p}(\rho(t),\phi(t))\ge0.
\end{equation}
\end{theorem}

The following Proposition gives a special solution to the Langevin deformation of flows \eqref{Lan} on $P_q(\mathbb{R}^n)$. 

\begin{proposition}\label{specialsolutionp}
Let $\alpha(t)=\frac{\dot{w}(t)}{w(t)}$ and $\beta(t)$ a smooth function on $(0, T]$ such that
\begin{align}\label{alphaeq}
&c^p\left(\dot{\alpha}(t)+\alpha^2(t)\right)+(p-1)\alpha(t)=\frac{p-1}{p^{q-1}}\frac{\alpha^{2-q}(t)}{w^q(t)},\\ \label{betaeq}
&c^p\dot{\beta}(t)+\beta(t)=n\log w(t)-\log c_{n,p}-1,
\end{align}
where $c_{n, p}$ is a constant given by \eqref{cpn}. For $x\in\mathbb{R}^n$, $t>0$, let 
\begin{equation}\label{RM for L flow1}
\rho_{c,n,p}(t,x):=c_{n,p}w(t)^{-n}\exp\left(-\frac{p-1}{p^q}\frac{\|x\|^q}{w(t)^q}\right), \qquad\phi_{c,n,p}(t,x):=\frac{\alpha(t)^{q-1}}q\|x\|^q+\beta(t).
\end{equation}
Then $(\rho_{c,n,p},\phi_{c,n,p})$ is a special solution to the Langevin deformation of flows \eqref{Lan} on $P_q(\mathbb{R}^n)$. Moreover,
\begin{align*}
{\rm Ent}(\rho_{c,n,p}(t))=&\int_{\mathbb{R}^n}\rho_{c,n,p}\log\rho_{c,n,p}\,dx=-\frac nq\left(1+\log\left(c_{n,p}^{-\frac qn}w^q(t)\right)\right),\\
{\rm I}(\rho_{c,n,p}(t))=&\int_{\mathbb{R}^n}|\nabla\phi_{c,n,p}|^{p-2}|\nabla\log\rho_{c,n,p}|^2_A\rho_{c,n,p}\,dx=\frac{p-1}{p^{q-1}}\frac{n\alpha^{2-q}(t)}{w^{q}(t)}.
\end{align*}
\end{proposition}

Note that, a direct calculation shows that, when $M=\mathbb R^n$,we have
\begin{equation*}
\frac1{\eta(t)}\frac{d}{dt}{W}_{c,n,p}(\rho_{c,n,p}(t),t)+\frac1{c^p}I_{c,n,p}(\rho_{c,n,p}(t),\phi_{c,n,p}(t))=0.
\end{equation*}
In view of this, $(\rho_{c,n,p}, \phi_{c,n,p})$ provides a rigidity model for the $W$-entropy-information associated with the Langevin deformation of flows \eqref{Lan} on $TP_q(M)$ over complete Riemannian manifolds with bounded geometry condition and with $\mathrm{Ric}\geq 0$. See Theorem \ref{RT_WIE} in Section \ref{LaHa} below.

%
%
%

\begin{remark} In the extremal cases $c=0$ and $c=\infty$, we have
\begin{itemize}
  \item [(1)] When $c=0$ in \eqref{Lan}, we have $\phi=-\log\rho-1$, and $\rho$ satisfies the 
  $p$-Laplacian heat equation 
  \begin{equation}\label{BPHE}
(p-1)^{1-p}\partial_t u^{p-1}=\Delta_{p}u:=\nabla\cdot(|\nabla u|^{p-2}\nabla u), \quad u=\rho^{1\over p-1}.
\end{equation}  
In this case, a special solution to \eqref{pode} and \eqref{alphaeq} in Proposition \ref{specialsolutionp} is given by 
  $$w(t)=t^{\frac1p}, \ \ \alpha(t)=\frac1{pt},$$
  and  a special solution to the $p$-heat equation \eqref{BPHE}
  on $\mathbb{R}^n$ is given by  
\begin{equation*}\label{RM for gra}
\rho_{0,n,p}(t,x)=c_{n,p}t^{-\frac np}\exp\left(-\frac{1}{q}\frac{\|x\|^q}{(pt)^{q-1}}\right),\quad \phi_{0,n,p}(t,x)=\frac{\|x\|^q}{q(pt)^{q-1}}+\frac np\log t-\log c_{n,p}-1.
\end{equation*}
 Thus
 $$
  \frac{d^2}{dt^2}{\rm Ent}(\rho)=p\int_M|\nabla \log\rho|^{2p-4}(|\nabla^2 \log\rho|^2_A+{\rm Ric}(\nabla \log\rho,\nabla \log\rho))\rho\,dv.
      $$
 By the definition of the $W$-entropy
 $$
 W_{0,n, p}(\rho(t),t):=\frac d{dt}\Big(t{\rm Ent}_{0,n, p}(\rho(t))\Big),
 $$
 the $W$-entropy-information formula \eqref{wcnentropy} can be rewritten as follows
 \begin{align*}
\frac1{t}\frac{d}{dt}{W}_{0,n, p}(\rho(t),t)
=p\int_M\left[\Big||\nabla \phi|^{p-2}\nabla_i\nabla_j\rho-\frac{1}{pt} a_{ij}\Big|^2_A+|\nabla \phi|^{2p-4}{\rm Ric}(\nabla \phi,\nabla \phi)\right]\rho\,dv,
\end{align*}
 which is equivalent to the $W$-entropy formula for the $p$-Laplacian heat equation \eqref{BPHE} on compact Riemannian manifold proved by  Kotschwar-Ni \cite{KoNi}.
 
 \item [(2)] When $c=\infty$ in  \eqref{Lan}, $(\rho(t),\phi(t))$ satisfies the $L^q$-geodesic flow equations \eqref{NDE0} on $TP_q(M)$.  In this case, a special solution to \eqref{pode} and \eqref{alphaeq} in Proposition \ref{specialsolutionp} is given by 
 $$w(t)=t, \ \ \ \ \alpha(t)=\frac{1} {t}, \ \ \eta(t)=t,$$ 
 and a special solution to \eqref{NDE0} on $TP_q(\mathbb{R}^n)$ is given by $(\rho_{\infty,n,p},\phi_{\infty,n,p})=(\rho_{n,p},\phi_{n,p})$ as in \eqref{special}. 
By the definition of the $W$-entropy
 $$
 W_{\infty,n, p}(\rho(t),t):=\frac d{dt}\Big(t{\rm Ent}_{\infty,n, p}(\rho(t))\Big),
 $$
the $W$-entropy-information formula \eqref{wcnentropy} can be rewritten as follows
\begin{align*}
\frac1{t}\frac{d}{dt}{W}_{\infty,n,p}(\rho(t),t)
=\int_M\left[\Big||\nabla \phi|^{p-2}\nabla_i\nabla_j\phi-\frac1ta_{ij} \Big|^2_A+|\nabla \phi|^{2p-4}{\rm Ric}(\nabla \phi,\nabla \phi)\right]\rho\,dv,
\end{align*}
which is the $W$-entropy formula \eqref{NLLWEnt} in Theorem \ref{CEntForm}.
\end{itemize}
\end{remark}

\section{Variational formulas for the geodesic flow on $L^q$-Wasserstein space}
Let $(M,g)$ be an $n$-dimensional  complete Riemannian manifold with bounded geometry condition. The $p$-Laplacian $\Delta_p$ is defined by 
\begin{equation*}
\Delta_p u:=\nabla\cdot\left(|\nabla u|^{p-2}\nabla u\right), \quad u\in C^\infty(M).
\end{equation*}
The linearization of the $p$-Laplacian $\Delta_p$ at $u\in C^2(M)$ with $\nabla u\neq0$ is given by (see e.g. \cite{KoNi})
\begin{equation*}\label{lin}
\mathcal{L}(\psi):= \nabla\cdot\left(|\nabla u|^{p-2}A(\nabla \psi)\right)
\end{equation*}
for $\psi\in C^\infty(M)$, where $A$ is the tensor defined in  \eqref{Atensor}.
Due to the degeneracy and singularity of the $p$-Laplacian on $\nabla u = 0$, an $\varepsilon$-regularization method is employed, which means that one replaces the linearized operator $\mathcal{L}$ by its approximating operator $\mathcal{L}_{{\varepsilon}}$ defined as 
$$
\mathcal{L}_{{\varepsilon}}\psi:=\nabla\cdot\left( w_{\varepsilon}^{\frac p2-1}A_{\varepsilon}(\nabla \psi)\right),
$$
where $\varepsilon>0$, $w_{\varepsilon}=|\nabla u_{\varepsilon}|^2+\varepsilon$ and $A_{\varepsilon}=g+(p-2)\frac{\nabla u_{\varepsilon}\otimes \nabla u_{\varepsilon}}{w_{\varepsilon}}$.  See \cite{KoNi}.

We first prove the entropy variational formula on $P_p(M)$. When $p=2$, it was due to Lott \cite{Lott}. 


\begin{proposition}\label{1stvar}
{Let $(\rho, \phi): [s_0-\epsilon,s_0+\epsilon]\times[0,1]\to C^\infty(M, \mathbb{R}^{+})\times C^{\infty}(M, \mathbb{R})$ be smooth functions satisfying the nonlinear transport equation 
\begin{equation}\label{NTE}
\frac{\partial \rho}{\partial t}+\nabla\cdot\left(\rho|\nabla \phi|^{p-2}\nabla \phi\right)=0,
\end{equation}
where for any fixed $t\in[0,1]$, $\phi(\cdot,t):[s_0-\epsilon,s_0+\epsilon]\to C^{\infty}(M)$. 
Let $s\mapsto c(s,\cdot)=\rho(s,\cdot)dv$ be a smooth curve in ${P}_q(M)$}. Define the energy functional as follows
$$
E(c(s)):= \frac1p\int^1_0\int_M|\nabla\phi(s,t)|^p\rho(s,t)\, dvdt.
$$
Then, the variation of $E(c(s))$ with respect to $s$ is given by
\begin{equation}\label{ELE}
\frac{d}{ds}E(c(s))=\frac1{p-1}\int_M\phi\frac{\partial\rho}{\partial s}dv\Big|^1_{t=0}-\frac1{p-1}\int^1_0\int_M\left(\frac{\partial \phi}{\partial t}+\frac1p|\nabla \phi|^p\right)\frac{\partial\rho}{\partial s}\,dv dt.
\end{equation}
\end{proposition}

\proof The proof is similar to the case $p=2$ in Lott \cite{Lott} and S. Li-Li \cite{LiLi3}. 
Direct calculation implies that
\begin{equation}\label{dtEc}
\frac{d}{ds}E(c(s))=\int^1_0\int_M\left(|\nabla\phi|^{p-2}\Big\langle\nabla\phi,\nabla\frac{\partial\phi}{\partial s}\Big\rangle \rho+\frac1p|\nabla\phi|^p\frac{\partial\rho}{\partial s}\right)dv dt.
\end{equation}
For fixed $h\in C^{\infty}(M)$,  from \eqref{NTE} and integration by parts, we have
$$
\int_Mh\frac{\partial\rho}{\partial t}dv=\int_M|\nabla\phi|^{p-2}\langle\nabla\phi,\nabla h\rangle\rho\, dv.
$$
Hence
$$
\int_Mh\frac{\partial^2\rho}{\partial s\partial t}dv
=\int_M\left(|\nabla\phi|^{p-2}\left\langle A\Big(\nabla\frac{\partial\phi}{\partial s}\Big),\nabla h\right\rangle\rho+|\nabla\phi|^{p-2}\langle\nabla h,\nabla\phi\rangle\frac{\partial\rho}{\partial s}\right) dv,
$$
where $A$ is defined in \eqref{Atensor}.
Taking $h=\phi$, we have
\begin{equation}\label{varpst}
\int_M\phi\frac{\partial^2\rho}{\partial s\partial t}dv
=\int_M\left((p-1)|\nabla\phi|^{p-2}\Big\langle\nabla \phi,\nabla\frac{\partial\phi}{\partial s}\Big\rangle\rho+|\nabla\phi|^{p}\frac{\partial\rho}{\partial s}\right) dv.
\end{equation}
Combining \eqref{dtEc} and \eqref{varpst}, we have 
\begin{align*}
\frac{d}{ds}E(c(s))=&\frac1{p-1}\int^1_0\int_M\left(\phi\frac{\partial^2\rho}{\partial s\partial t}-\frac1p|\nabla\phi|^{p}\frac{\partial\rho}{\partial s}\right)dv dt\\
=&\frac1{p-1}\int^1_0\int_M\left(\frac{\partial}{\partial t}\Big(\phi\frac{\partial\rho}{\partial s}\Big)-\Big(\frac{\partial\phi}{\partial t}+\frac1p|\nabla\phi|^{p}\Big)\frac{\partial\rho}{\partial s}\right)dv dt,
\end{align*}
from which the variational formula \eqref{ELE} holds.
\endproof

From \eqref{ELE}, the Euler-Lagrange equation for $E$ is given by the $p$-Hamitlon-Jacobi equation
\begin{equation*}
   \frac{\partial \phi}{\partial s}+\frac1p|\nabla \phi|^p=0.
\end{equation*}
Thus, if a $L^q$-geodesic flow $(\rho(t), \phi(t), t\in [0, T])$  is a smooth curve in $P_q(M)$, then it satisfies \eqref{NDE0}.

\begin{proposition}\label{Evo1} Let $(\rho,\phi)$ be a smooth solution to the $L^q$-geodesic flow equation  \eqref{NDE0}. Then 
\begin{align*}
\frac{d}{dt}\int_M\phi\rho\,dv=&\frac1q\int_M|\nabla\phi|^p\rho\,dv,\\
\frac{d^2}{dt^2}\int_M\phi\rho\,dv=&\frac1q\frac{d}{dt}\int_M|\nabla\phi|^p\rho\,dv=0.
\end{align*}
\end{proposition}
\proof
By \eqref{NDE0}, direct calculation implies that
\begin{align*}
\frac{d}{dt}\int_M\phi\rho\,dv=&\int_M(\partial_t\rho\phi+\rho\partial_t\phi)\,dv
=-\int_M\nabla\cdot(\rho|\nabla\phi|^{p-2}\nabla\phi)\phi\,dv-\frac1p\int_M|\nabla\phi|^p\rho\,dv=\frac1q\int_M|\nabla\phi|^p\rho\,dv,
\end{align*}
and
\begin{align*}
\frac1q\frac{d}{dt}\int_M|\nabla\phi|^p\rho\,dv=&\frac1q\int_M|\nabla\phi|^p\partial_t\rho\,dv
+\frac pq\int_M|\nabla\phi|^{p-2}\langle\nabla\phi,\nabla\partial_t\phi\rangle\rho\,dv\\
=&-\frac1q\int_M|\nabla\phi|^{p}\nabla\cdot(\rho|\nabla\phi|^{p-2}\nabla\phi)\,dv
-\frac 1q\int_M|\nabla\phi|^{p-2}\langle\nabla\phi,\nabla|\nabla\phi|^p\rangle\rho\,dv
=0.
\end{align*}
\endproof

\begin{proposition}\label{EDF}
 Let $(M,g)$ be a complete Riemannian manifold with bounded geometry condition.  Let $(\rho,\phi)$ be a smooth solution to the $L^q$-geodesic equations \eqref{NDE0} satisfying the following growth condition 
 $$
 \int_M\left[|\nabla\log\rho|^p+|\nabla\phi|^{p}+|\nabla\phi|^{2p-2}+|\nabla^2\phi|_A^{2p-2}+|\Delta_p\phi|\right]\rho\, dv<\infty. 
 $$
 Assume there exist a point $o \in M$, and some functions $C_i, \alpha_i \in C\left([0, T], \mathbb{R}^{+}\right)$, $i=1, 2$, such that
$$
C_1(t) e^{-\alpha_1(t) d^q(x, o)} \leq \rho(t, x) \leq C_2(t) e^{\alpha_2(t) d^q(x, o)}, \quad \forall x \in M, t \in[0, T],
$$
and
$$
\int_M d^{pq}(x, o) \rho(t, x) d v(x)<\infty, \quad \forall t \in[0, T] .
$$
 Then  the following variational formulas hold:
\begin{align}\label{HJEnt1}
\frac{d}{d t}{\rm Ent}(\rho)=&\int_M|\nabla \phi|^{p-2}\langle\nabla \phi,\nabla\rho\rangle \, dv=-\int_M\rho\Delta_{p}\phi\,dv,
\end{align}
\begin{align}\label{HJEnt2}
\frac{d^2 }{d t^2}{\rm Ent}(\rho)=&\int_M|\nabla \phi|^{2p-4}\left(|\nabla^2 \phi|_A^2+{\rm Ric}(\nabla \phi,\nabla \phi)\right)\rho\,dv,
\end{align}
where $A$ is  defined in \eqref{Atensor},  and $\langle X,Y\rangle_A=\sum_{ij}A^{ij}X_iY_j$ for all $X, Y\in C^\infty(\Gamma(TM))$.
\end{proposition}
\begin{proof}[Proof] 
Let $\eta_k$ be an increasing sequence of functions in $C_0^{\infty}(M)$ such that $0 \le \eta_k \le 1$,
$\eta_k = 1$ on $ B(o, k)$ , $\eta_k = 0$ on $M \backslash B(o, 2k)$, and 	$\eta_k	 \le {1\over k}$. 
Let $(\rho, \phi)$ be a smooth solution to Eq. \eqref{NDE0}. Integrating by parts, we have
\begin{align*}
\frac{d }{d t}\int_M(\rho\log\rho) \eta_kdv
=&\int_M|\nabla \phi|^{p-2}\langle\nabla \phi,\nabla\rho\rangle \eta_k\,dv+\int_M|\nabla \phi|^{p-2}\langle\nabla \phi,\nabla\eta_k\rangle(1+\log\rho)\rho\,dv\\
:=&I_1(k)+I_2(k).
\end{align*}
Under the assumption of proposition, we have
$$
\int_M|\nabla\phi|^p\rho dv<\infty,\quad\int_M|\nabla\log\rho|^p\rho dv<\infty.
$$
By H\"older's inequality,
\begin{align*}
\int_M|\nabla \phi|^{p-2}\langle\nabla \phi,\nabla\log\rho\rangle\rho\,dv
\leq&\left(\int_M|\nabla\phi|^p\rho dv\right)^{\frac1q}\left(\int_M|\nabla\log\rho|^p\rho dv\right)^{\frac1p}<\infty. 
\end{align*}
Hence $|\nabla \phi|^{p-2}\langle\nabla \phi,\nabla\rho\rangle|\in L^1(M)$. By the Lebesgue dominated convergence
theorem, as $k\rightarrow \infty$, we have
\begin{equation}\label{I1k1}
I_1(k)=\int_M|\nabla \phi|^{p-2}\langle\nabla \phi,\nabla\rho\rangle \eta_k\,dv\rightarrow \int_M|\nabla \phi|^{p-2}\langle\nabla \phi,\nabla\rho\rangle\,dv.
\end{equation}
On the other hand, as $k\rightarrow \infty$, we have 
\begin{align}\label{I1k2}
I_1(k)=&-\int_M\nabla\cdot(\eta_k|\nabla \phi|^{p-2}\nabla \phi)\rho\,dv=-\int_M(\Delta_p\phi)\rho\eta_k\,dv-\int_M|\nabla\phi|^{p-2}\nabla\phi\cdot\nabla\eta_k\rho\,dv\notag\\
\to&-\int_M(\Delta_p\phi)\rho\,dv.
\end{align}
Similarly, noticing $|\nabla\eta_k| \le 1/k$. Under the assumption of proposition, the Lebesgue dominated convergence theorem yields  
\begin{equation}\label{I2k}
I_2(k)=\int_M|\nabla \phi|^{p-2}\langle\nabla \phi,\nabla\eta_k\rangle(1+\log\rho)\rho\,dv\to0 \quad \text{ as }k\to\infty.
\end{equation}
Combining \eqref{I1k1}, \eqref{I1k2} with \eqref{I2k}, we complete the proof of \eqref{HJEnt1}.

By the  $p$-Bochner formula  (see \cite{KoNi})
\begin{equation}\label{p-bochner}
\mathcal{L}(|\nabla \phi|^p)=p|\nabla \phi|^{2p-4}(|\nabla^2 \phi|^2_A+{\rm Ric}(\nabla \phi,\nabla \phi))+p|\nabla \phi|^{p-2}\langle\nabla \phi,\nabla \Delta_{p}\phi\rangle, 
\end{equation}
and integrating by parts, we have 
\begin{small}
\begin{align*}
&\frac{d }{d t}\int_M|\nabla \phi|^{p-2}\langle\nabla \phi,\nabla\rho\rangle \eta_k\, dv\notag=\int_M\left\langle\partial_t(|\nabla \phi|^{p-2}\nabla \phi),\nabla\rho\right\rangle\eta_k+\left\langle|\nabla \phi|^{p-2}\nabla \phi,\nabla\partial_t\rho\right\rangle\eta_k\,dv\notag\\
=&\int_M|\nabla \phi|^{2p-4}(|\nabla^2 \phi|^2_A+{\rm Ric}(\nabla \phi,\nabla \phi))\rho\eta_k\,dv+\frac1p\int_M|\nabla \phi|^{p-2}\Big\langle\nabla |\nabla\phi|^p,\nabla\eta_k\Big\rangle_A\rho\,dv-\int_M(\Delta_p\phi)|\nabla\phi|^{p-2}\langle\nabla\phi,\nabla\eta_k\rangle\rho\,dv\notag\\
:=&I_3(k)+I_4(k)+I_5(k),
\end{align*}
\end{small}
where we used the facts
\begin{align*}\label{ptphi}
\int_M\langle\partial_t(|\nabla \phi|^{p-2}\nabla \phi),\nabla\rho\rangle\eta_k\,dv
=&\int_M|\nabla \phi|^{p-2}\langle\nabla \partial_t \phi,\nabla\rho\rangle_A\eta_k\,dv,
\end{align*}
and
$$
-\int_M|\nabla \phi|^{p-2}\Big\langle\nabla |\nabla\phi|^p,\nabla\rho\Big\rangle_A\eta_k\,dv=\int_M\mathcal{L}(|\nabla\phi|^p)\rho\eta_k\,dv+\int_M|\nabla \phi|^{p-2}\Big\langle\nabla |\nabla\phi|^p,\nabla\eta_k\Big\rangle_A\rho\,dv.
$$
By $|\mathrm{Ric}|\le C$, under the assumption
$\int_M[|\nabla\phi|^{2p-2}+|\nabla^2\phi|^{2p-2}_A]\rho\,dv<\infty$, we have
$$
\int_M\left||\nabla \phi|^{2p-4}(|\nabla^2 \phi|^2_A+{\rm Ric}(\nabla \phi,\nabla \phi))\right|\rho\,dv\\
\le\int_M|\nabla \phi|^{2p-2}(|\nabla^2 \phi|^{2p-2}_A+C)\rho\,dv<\infty.
$$
Using the fact $0 \le\eta_k \le 1$ and $\eta_k \to 1$, the Lebesgue dominated convergence theorem yields
\begin{equation}\label{I3k}
I_3(k)\to\int_M|\nabla \phi|^{2p-4}(|\nabla^2 \phi|^2_A+{\rm Ric}(\nabla \phi,\nabla \phi))\rho\,dv.
\end{equation}
Using again the assumption $\int_M[|\nabla\phi|^{2p-2}+|\nabla^2\phi|^{2p-2}_A]\rho\,dv<\infty$,  we have
\begin{equation*}\label{I4k}
I_4(k)=\frac1p\int_M|\nabla \phi|^{p-2}\Big\langle\nabla |\nabla\phi|^p,\nabla\eta_k\Big\rangle_A\rho\,dv\le(p-1)\int_M|\nabla\phi|^{2p-4}|\nabla^2\phi\nabla\phi|\cdot|\nabla\eta_k|\rho\,dv\to0. 
\end{equation*}
Under the assumption of proposition, we have
$\int_M[(\Delta_p\phi)^p+|\nabla\phi|^p]\rho\,dv<\infty$. 
Using again the fact $0 \le\eta_k \le 1, \eta_k \to 1$ and $|\nabla\eta_k | \le \frac1k$, the Lebesgue dominated
convergence theorem yields
\begin{align}\label{I5k}
I_5(k)=-\int_M(\Delta_p\phi)|\nabla\phi|^{p-2}\langle\nabla\phi,\nabla\eta_k\rangle\rho\,dv\to0.
\end{align}
Combining \eqref{I3k}, \eqref{I4k} with \eqref{I5k}, we complete the proof of \eqref{HJEnt2}.

\end{proof}

\section{$W$-entropy formula for the geodesic flow on  $L^q$-Wasserstein space}
Applying the entropy variational formulas in Proposition \ref{EDF},  we can derive the ${W}$-entropy formula for the geodesic flow \eqref{NDE0} on the $L^q$-Wasserstein space $P_q(M)$. 
\begin{proof}[\bf{Proof of Theorem \ref{CEntForm}}]
By the definition of the ${W}$-entropy \eqref{WEntropy} and the entropy variational formulas in Proposition \ref{EDF}, we obtain
\begin{align}\label{dtwpn}\small
\frac{d}{dt}{W}_{n,p}(\rho,t)=&2\frac{d}{dt}{\rm Ent}_{n,p}(\rho,t)+t\frac{d^2}{dt^2}{\rm Ent}_{n,p}(\rho,t)\notag\\
=&t\int_M\left(\left||\nabla \phi|^{p-2}\nabla_i\nabla_j \phi-\frac{a_{ij}}{t}\right|_A^2+|\nabla \phi|^{2p-4}{\rm Ric}(\nabla \phi,\nabla \phi)\right)\rho\,  dv+\frac nt\notag\\
&+2\int_M|\nabla \phi|^{p-2}\langle\nabla \phi,\nabla\rho\rangle \, dv+2\int_M(\Delta_{p}\phi)\rho\,dv-\frac{n}{t}\notag\\
=&t\int_M\left(\left||\nabla \phi|^{p-2}\nabla_i\nabla_j \phi-\frac{a_{ij}}{t}\right|_A^2+|\nabla \phi|^{2p-4}{\rm Ric}(\nabla \phi,\nabla \phi)\right)\rho\,dv,
\end{align}
where $a_{ij}$ is the inverse matrix of $A^{ij}$ and we used the identity 
\begin{equation*}\label{ID}
{\rm tr}_A(|\nabla \phi|^{p-2}\nabla^2 \phi)=|\nabla \phi|^{p-2}(A^{ij}\nabla_i\nabla_j\phi) =\Delta_p\phi.
\end{equation*}

The rigidity part can be proved as follows. Indeed, under the assumption
${\rm Ric}\ge 0$, if  $\frac{d}{dt}W_{n,p}(\rho,\phi)=0 $ holds at some $t = t_0 > 0$, then the $W$-entropy formula \eqref{NLLWEnt} yields
\begin{equation*}\label{rigeq}
|\nabla \phi|^{p-2}\nabla_i\nabla_j \phi=\frac{a_{ij}}{t}, 
\end{equation*}
which is equivalent to 
\begin{equation*}\label{rigeq2}
\nabla_i\nabla_j \phi=\frac1{t|\nabla \phi|^{p-2}}\left(g_{ij}+(q-2)\frac{\nabla_i\phi\nabla_j\phi}{|\nabla \phi|^{2}}\right). 
\end{equation*}
By the Theorem 6.19 of Kotschwar-Ni in \cite{KoNi}, we can conclude that $M$ is isometric to $\mathbb{R}^n$ and $(\rho,\phi)=(\rho_n,\phi_n)$.

\end{proof}

In the case $p=q=2$, S. Li and the second named author \cite{LiLi3, LiLi4} observed that
\begin{equation*} 
\frac{d}{dt}W_n(\rho,t)=\frac{d^2}{dt^2}(tH_n(\rho(t)))=\frac{d^2}{dt^2}\big(t{\rm Ent}(\rho(t))+nt\log t\big).
\end{equation*}
As a corollary of Theorem \ref{LLinfty}, they derived the following convexity and rigidity theorem. The convexity part was due to Lott \cite{Lott}. 

\begin{theorem} [Lott \cite{Lott}, S. Li-Li \cite{LiLi3, LiLi4}] \label{Lott2} The function $t\mapsto \mathcal{E}(\rho(t)):=t{\rm Ent}(\rho(t))+nt\log t$ is convex  along the $L^2$-geodesic flow $(\rho(t), \phi(t))$  on the $L^2$-Wasserstein space $P_2(M)$ over a Riemannian manifold $(M, g)$ with non-negative Ricci curvature. Moreover, under the assumption ${\rm Ric}\ge 0$, ${d^2\over dt^2} \mathcal{E}(\rho(t))=0$ at some $t=t_0>0$ if and only if $M$ is isometric to $\mathbb R^n$ and $(\rho, \phi)=(\rho_n, \phi_n)$ as given in \eqref{phiminfty}. 
\end{theorem}


Similarly, we can prove  the following convexity and rigidity theorem for the $L^q$-geodesic flow on $P_q(M)$.  

\begin{theorem} For $q>1$, the function $t\mapsto  \mathcal{E}(\rho(t))=t{\rm Ent}(\rho(t))+nt\log t$ 
is convex along the  $L^q$-geodesic flow on  the $L^q$-Wasserstein space $P_q(M)$ over a Riemannian manifold $(M, g)$ with non-negative Ricci curvature. Moreover, under the assumption ${\rm Ric}\ge 0$, ${d^2\over dt^2} \mathcal{E}(\rho(t))=0$ at some $t=t_0>0$ if and only if $M$ is isometric to $\mathbb R^n$ and $(\rho, \phi)=(\rho_{n, p}, \phi_{n, p})$ as given in Theorem \ref{CEntForm}. 
\end{theorem}
\proof
Indeed, by \eqref{dtwpn} and a direct calculation 
\begin{align*}\label{Wconvex}
\frac{d^2}{dt^2}\mathcal{E}(\rho(t))
=&t\frac{d^2}{dt^2}{\rm Ent}(\rho)+2\frac{d}{dt}{\rm Ent}(\rho)+\frac nt \notag\\
=&2\int_M|\nabla \phi|^{p-2}\langle\nabla \phi,\nabla\rho\rangle \, dv+
t\int_M|\nabla \phi|^{2p-4}\left(|\nabla^2 \phi|_A^2+{\rm Ric}(\nabla \phi,\nabla \phi)\right)\rho\,  dv+\frac nt\notag\\
=&t\int_M\left(\left||\nabla \phi|^{p-2}\nabla_i\nabla_j \phi-\frac{a_{ij}}{t}\right|_A^2+|\nabla \phi|^{2p-4}{\rm Ric}(\nabla \phi,\nabla \phi)\right)\rho\,dv\notag\\
=&\frac{d}{dt}{W}_{n,p}(\rho,t).
\end{align*}
Thus, Theorem 4.2 follows from Theorem \ref{CEntForm}. 
\endproof

\section{Local existence and uniqueness of Langevin deformation}\label{LDF}
In this section, we prove the local existence and uniqueness of solution to the Cauchy problem for the Langevin deformation of flows \eqref{Lan} for fixed $c\in(0,\infty)$. 

To simplify notations and to avoid technical complexity, we only give the proof on Euclidean spaces and we would like to point out that the proof can be extended to compact Riemannian manifolds by applying the following Sobolev embedding theorem. Let $M$ be $\mathbb R^n$ or an $n$-dimensional compact Riemannian manifold. By \cite{Bre, Aubin}, there exists a constant $C_{\text {Sob }}>0$ such that
$$
\|f\|_{\frac{2 n}{n-2}} \leq C_{\mathrm{Sob}}\left(\|\nabla f\|_2+\|f\|_2\right), \quad \forall f \in C^{\infty}(M) .
$$
For $0 < \alpha < 1$, define the H\"older space $C^{0,\alpha}(M)$ equipped with the norm 
$$
\|f\|_{C^{0, \alpha}}:=\|f\|_{\infty}+\sup _{x, y \in M, x \neq y} \frac{|f(x)-f(y)|}{d(x, y)^\alpha},
$$
and define the Sobolev space $H^s(M)$ equipped with the Sobolev norm
$$
\|f\|_{s, 2}=\left(\sum_{|\beta|\leq s}\left\|D^\beta f\right\|_2^2\right)^{1 / 2},
$$
where $\beta$ is a multi-index. For any $\alpha \in(0,1)$, if $s>\alpha+\frac{n}{2}$, then the Bernstein-Rellich-Kondrakov-Morrey continuity embedding theorem holds
\begin{equation}\label{KEmb}
\|f\|_{C^{0, \alpha}} \leq C_{\alpha}\|f\|_{s, 2}.
\end{equation}
In particular, there exists a constant $C>0$ such that for any $s>\frac{n}{2}$, we have
\begin{equation}\label{Sobineq}
  \|f\|_{\infty} \leq C\|f\|_{s, 2} .
\end{equation}


Now we consider $\mathcal{V}(\rho)=\int_M \rho\log\rho dv$. Let $U=(\log\rho,u)^{T}=(\log\rho,v_q)^{T}:M\times[0,T]\to \mathbb R^{n+1}$. Then we can rewrite the Cauchy problem for $\eqref{NDEC3}$ with initial value $(\rho_0, u_0)$ as the following symmetric hyperbolic system
\begin{equation}\label{sym hyper}
\left\{\begin{array}{l}
    A_0^c(U)\partial_tU+\sum_{j=1}^{n}A_j^c(U)\partial_jU+BU=0,\\
    U(0,x)=U_0(x)=(\log\rho_0,u_0)(x),
\end{array}\right.
\end{equation}
where $A_0^c(U)=\mathrm{diag}\{1,c^p,\cdots,c^p\}$, $B=\mathrm{diag}\{0, 1,\cdots,1\}$, 
and for $j\in\{1,2,\cdots, n\}$, $A_j^c(U)$ is an $(n+1)\times(n+1)$ matrix with the entries $\left(A_j^c(U)\right)_{1,1}=|u|^{p-2}u^j$, $\left(A_j^c(U)\right)_{1,j+1}=\left(A_j^c(U)\right)_{j+1,1}=1$, $\left(A_j^c(U)\right)_{k,k}=c^p|u|^{p-2}u^j$ for $2\leq k\leq n+1$, and $\left(A_j^c(U)\right)_{k,l}=0$ otherwise.


Applying  Theorem 2 in \cite{Kato1975} by Kato, we obtain the following local existence and uniqueness of solution to the Cauchy problem for the symmetric hyperbolic system \eqref{sym hyper}. 
\begin{theorem}\label{Kato thm}
Let $c\in (0, \infty)$ and $M=\mathbb R^n$ or an $n$-dimensional compact Riemannian manifold.  Let $s$ be an integer such that $s>\frac{n}{2}+1$. Suppose $p\geq 2$. Then there exists a bounded open subset $D$ of $H^{s}(M; \mathbb R^{n+1})$ such that for any $U_0\in D$, there exists a unique solution $U$ of \eqref{sym hyper} defined on $[0,T]$ for some $T>0$ and 
\begin{equation*}
U\in C\left([0,T];D\right)\cap C^1\left([0,T];H^{s-1}(M; \mathbb R^{n+1})\right).
\end{equation*} 
More precisely, if we take $U_{00}=(\rho_{00},u_{00})$ such that $U_{00}\in H^s(M;\mathbb R^{n+1})\cap C_c^\infty(M;\mathbb R^{n+1})$ and $\rho_{00}\geq \delta_1>0$, $|u_{00}|\geq \delta_2>0$, then the open subset $D$ can be taken as 
$$D=\left\{U=(\log\rho,u): \|U-U_{00}\|_{H^s}<K\right\}.$$   
\end{theorem}



\begin{proof}
According to the Theorem 2 in \cite{Kato1975}, we consider the operator $G_j(t): D\to H^s_{ul}(M; \mathbb R^{(n+1)\times(n+1)})$ defined by $G_j(t)[U]=A_j^c(U)$ for $j=1,\cdots,n$, where $H^s_{ul}(M)$ is the uniformly local Sobolev space defined in \cite{Kato1975} and 
\begin{equation*}
    \|u\|_{s,ul}:=\|u\|_{H^s_{ul}(M)}=\sup_{|\alpha|\leq s}\sup_{x\in M}\left\{\int_{d(y,x)<1}|D^\alpha u(y)|^2dy\right\}^{1\over 2}. 
\end{equation*}
Now we only have to verify the coefficients $\{A_j^c$, $j=1,\cdots, n\}$ satisfy the uniformly boundness and Lipschitz condition since $A_0^c$ and $B$ are constant matrices. First, we verify 
\begin{equation}\label{uni bdd}
\sup_{U\in D}\|A_j^c(U)\|_{H^s_{ul}}\leq C, \quad j=1,\cdots,n.
\end{equation} 
Take $U=(\log\rho,u)\in D$ and denote $f(u)=|u|^{p-2}u$ with $p\geq 2$. 
By the Lagrange mean value theorem, there exists some $\theta\in[0,1]$ and   
$\xi=\theta u_{00}+(1-\theta)u$ such that 
\begin{equation*}
    f(u)=f(u_{00})+\nabla f(\xi)(u-u_{00}), 
\end{equation*}
where $\nabla f(\xi)=(p-2)|\xi|^{p-4}\xi\otimes \xi+|\xi|^{p-2}I$. 

By the Sobolev inequality \eqref{Sobineq}, if $s> {n\over 2}+1$, then we have 
\begin{equation*}
    \|u\|_{L^\infty}\leq \|U\|_{L^\infty}\leq C_s\|U\|_{H^s}\leq C^\prime, \quad \forall U\in D.
\end{equation*}
Note that $\sup_{U\in D}\|\xi\|_{\infty}\leq \sup_{\theta\in[0,1]}\sup_{U\in D}\|\theta u_{00}+(1-\theta)u\|_{H^s}\leq C^{\prime\prime}$. 
Thus we have 
\begin{equation*}
\begin{aligned}
\|f(u)\|_{H^s}
&\leq \|f(u_{00})\|_{H^s}+\|\nabla f(\xi)\|_\infty\|u-u_{00}\|_{H^s}\leq \|f(u_{00})\|_{H^s}+C^{\prime\prime\prime}\|u-u_{00}\|_{H^s}. 
\end{aligned}
\end{equation*}
Choosing $U_{00}\in H^s(M;\mathbb R^{n+1})\cap C_c^\infty(M;\mathbb R^{n+1})$ such that $|u_{00}|\geq \delta_2>0$ and noticing that $s>{n\over 2}+1$, we can verify that $\|f(u_{00})\|_{H^s}$ is finite. Thus we have $\|f(u)\|_{H^s}<+\infty$. Then we obtain \eqref{uni bdd}.

Next we verify the Lipschitz condition. This requires that there exists a constant $L>0$ such that 
\begin{equation}\label{Lip cond}
    \|A_j^c(U)-A_j^c(V)\|_{L^2_{ul}}\leq L\|U-V\|_{L^2}, \quad \forall U,V\in D.
\end{equation}
Noticing that $\|\nabla f(u)\|_{L^\infty}\leq C^{\prime\prime}$, we have
\begin{equation*}
\begin{aligned}
\|A_j^c(U)-A_j^c(V)\|_{L^2_{ul}}^2 &\leq  (n+1)\max\{1,c\}\int_{\mathbb R^n}\left|f(u)-f(v)\right|^2 dx
\leq (n+1)\max\{1,c\}C^{\prime\prime}\|U-V\|_{L^2}^2,
\end{aligned}
\end{equation*}
where $U=(\log\rho_1,u)$ and $V=(\log\rho_2, v)$. Thus we obtain \eqref{Lip cond} by taking $L=(n+1)\max\{1,c\}C^{\prime\prime}$. 
\end{proof}

Let $k=[s-\frac{n}{2}]$, where $[x]$ denotes the integer part of $x$ and $s>\frac{n}{2}+1$. By applying the Sobolev embedding theorem again, we have 
\begin{corollary}\label{smooth rho-u}
Let $M$ be $\mathbb R^n$ or an $n$-dimensional compact Riemannian manifold. Let $c\in(0,+\infty)$ and $p\geq 2$. Suppose that $(\rho_0,u_0)\in H^s(M;\mathbb R^{n+1})\cap D$. Then, there exists a constant $T > 0 $ such that the Cauchy problem for the compressible $p$-Euler equation with damping $\eqref{NDEC3}$, where $V(\rho)=\rho\log\rho$, has a unique smooth solution $(\rho,u)$ in $C^1([0, T],C^k(M)\times C^k(M))$.
\end{corollary}

Now we turn back to the Langevin deformation $\eqref{Lan}$. We need to prove that if the initial value $u(0, \cdot)=\nabla \phi(0, \cdot)$ for some smooth function $\phi(0, \cdot)$, then $u(t, \cdot)$ will keep the gradient structure along $t>0$. To see this, we show the following result.

\begin{theorem}\label{close form}
Let $c \in(0, \infty)$ and $p\geq 2$. Let $M$ be $\mathbb{R}^n$ or an $n$-dimensional compact Riemannian manifold. Let $(\rho, u)$ be the smooth solution to the compressible $p$-Euler equation with damping $\eqref{NDEC3}$. Let $u^*\in\Gamma(\Lambda^1T^*M)$ be the dual of $u$ and denote $\omega=d u^*$. Then
\begin{equation}\label{eq of omega}
\partial_t \omega+d\left(|u|^{p-2}\nabla_u u^*\right)=-c^{-p} \omega .
\end{equation}
Moreover, if $\|\nabla^k u\|_{L^\infty}\leq C$ for $k=0, 1, 2$, then for all $t \in[0, T]$, there exists a constant $C^\prime$ depending on $C$ such that 
$$
\|\omega(t)\|_p \leq\|\omega(0)\|_p e^{\left(C^\prime-c^{-2}\right) t} .
$$
In particular, if $u^*(0, \cdot)$ is a closed form, so is $u^*(t, \cdot)$, i.e., $d u^*(0, \cdot)=0$ implies $d u^*(t, \cdot)=$ 0 on $[0, T]$. Furthermore, the Poincar\'e lemma gives that $u^*$ is locally exact. i.e, there exists a smooth function $\phi$ such that $u=\nabla\phi$ on $t\in[0,T]$.     
\end{theorem}

\begin{proof} 
Taking exterior differentiation on the both sides of the second equation in $\eqref{NDEC3}$, letting $\omega=du^*$ and using $d^2=0$, we obtain \eqref{eq of omega}. By the proof of Theorem 4.3 in \cite{LiLi3}, we have 
\begin{equation*}
    d\nabla_u u^*-\nabla_u du^*=\sum_{i=1}^n du_i\wedge \nabla_{e_i}u^*,
\end{equation*}
where $\left\{e_i\right\}_{i=1}^n$ is a local orthonormal frame. Notice that 
\begin{equation*}
d\omega=\sum_{i,j=1}^n\omega_{ij} ~e_i^*\wedge e_j^*=\sum_{i,j=1}^ne_i(u_j)e_i^*\wedge e_j^*.
\end{equation*}
Thus we have 
\begin{align*}
d\left(|u|^{p-2}\nabla_u u^*\right)
=&\sum_{i,j=1}^n u_j\left(\nabla_{e_i}|u|^{p-2}\right)e_i^*\wedge\nabla_{e_j}u^*+|u|^{p-2}\left(\nabla_u du^* +\sum_{i=1}^n du_i\wedge \nabla_{e_i}u^*\right)\\
=&|u|^{p-2}\nabla_u \omega+ |u|^{p-4} \sum_{i,j,k=1}^n \left[(p-2)M_{ij}+|u|^2\omega_{ij}\right]\omega_{jk} ~e_i^*\wedge e_k^*\\
=&|u|^{p-2}\nabla_u \omega+I,
\end{align*}
where 
\begin{equation*}
    M_{ij}=\sum_{l=1}^n \omega_{il} u_l u_j,\quad I=|u|^{p-4} \sum_{i,j,k=1}^n \left[(p-2)M_{ij}+|u|^2\omega_{ij}\right]\omega_{jk} ~e_i^*\wedge e_k^*\in \Lambda^2(T^*M).
\end{equation*}
Noticing that
\begin{equation*}
    |I|\leq \|u\|_{L^\infty(M,\mu)}^{p-4}\max\left\{(p-2)|M|,\|u\|_{L^\infty(M,\mu)}^2|\omega|\right\}|\omega|, 
\end{equation*}
where $|\cdot|=\|\cdot\|_{\rm HS}$, there exists a constant $C^\prime=C^\prime(\|u\|_{L^\infty(M,\mu)}, \|M\|_{L^\infty(M,\mu)}, \|\omega\|_{L^\infty(M,\mu)})$, such that 
\begin{equation*}
    |I|\leq C|\omega|.
\end{equation*}
Taking inner product with $|\omega|^{\gamma-2}\omega$ on the both sides of $\eqref{eq of omega}$, where $\gamma\geq 2$ is a constant, and integrating on $M$, we have
\begin{equation*}
\int_{M}\left\langle \frac{D}{dt}\omega, |\omega|^{\gamma-2}\omega\right\rangle dv+\int_M \left\langle I, |\omega|^{\gamma-2}\omega\right\rangle dv=-c^{-p}\int_M \left\langle \omega, |\omega|^{\gamma-2}\omega\right\rangle dv,
\end{equation*}
where $\frac{D}{dt}\omega=\partial_t\omega+|u|^{p-2}\nabla_u \omega$. That is 

\begin{equation*}
\begin{aligned}
\frac{1}{\gamma}\frac{d}{dt}\int_{M}\left|\omega\right|^\gamma dv=&-\int_M \left\langle I, |\omega|^{\gamma-2}\omega\right\rangle dv-c^{-p}\int_{M}\left|\omega\right|^\gamma dV\leq \left(C^\prime-c^{-p}\right)\int_{M}\left|\omega\right|^\gamma dv.
\end{aligned}
\end{equation*}
Then by the Gronwall's inequality, we have
\begin{equation*}
\|\omega(t)\|_{L^\gamma(M)}\leq e^{\left(C^\prime-c^{-p}\right)t}\|\omega(0)\|_{L^\gamma(M)}.
\end{equation*}
Moreover, if $\omega(0)=du^*(0)=dd\phi_0=0$, then $\omega(t)=du^*=0$ for all $t\in[0,T]$. By the Poincar\'e lemma, $u(t, \cdot)$ will keep the gradient form along $t>0$. 
\end{proof}

Now we state the local existence and uniqueness to the Cauchy problem for the Langevin deformation  $\eqref{Lan}$ for any fixed $c\in(0, +\infty)$.

\begin{theorem} Let $M$ be $\mathbb R^n$ or an $n$-dimensional compact Riemannian manifold. Let $c\in(0,+\infty)$ and $p\geq 2$. Suppose that $(\rho_0,\phi_0)\in \bigcap\limits_{s>\frac{n}{2}+1}H^s(M, \mathbb{R}^{+})\times \bigcap\limits_{s>\frac{n}{2}+2}H^s(M, \mathbb{R})$ with $\rho_0>0$. Then, there exists a constant $T > 0 $ such that the Cauchy problem for the Langevin deformation $\eqref{Lan}$ has a unique smooth solution $(\rho,\phi)$ in $C^1([0, T],C^\infty(M,  \mathbb{R}^{+})\times C^\infty(M, \mathbb{R}))$.
\end{theorem}

\begin{proof}
The proof is similar to \cite{LiLi3} for $p=2$. Since we obtained the local existence and uniqueness of smooth solution to the compressible $p$-Euler equation with damping in Corollary \ref{smooth rho-u}, we can construct
\begin{equation*}
\phi(t,x)=e^{-\frac{t}{c^p}}\phi_0(x)-e^{-\frac{t}{c^p}}\int_0^t e^{-\frac{s}{c^p}}\left(\frac{V^\prime(\rho)(s,x)}{c^p}+\frac{1}{p}|u(s,x)|^p\right)ds.
\end{equation*}
Combining with Theorem $\ref{close form}$, we complete the proof. \end{proof}

\begin{remark}
The local existence and uniqueness result to the Cauchy problem for the Langevin
deformation $\eqref{Lan}$ can be extended to complete
Riemannian manifolds on which the Bernstein-Rellich-Kondrakov-Morrey continuity embedding theorem \eqref{KEmb} holds. When $p=2$,  the global existence and uniqueness of smooth solution with small initial data to the compressible Euler equations with damping on $\mathbb R^n$ were established by Wang and Yang \cite{WY}. See also Sideris, Thomases and Wang \cite{STW} for alternative approach. Assuming that $M$ is a compact Riemannian manifold, S. Li and the second named author \cite{LiLi4} proved that if the initial datum has small Sobolev norm, then the Cauchy problem for the Langevin deformation $\eqref{LLflow}$ on $TP_2(M)$ has a global unique solution in $H^s$ with $s> {n\over 2}+1$ for any fixed $c\in(0, +\infty)$. Convergence results for both $c\to 0$ and $c\to\infty$ were proved in \cite{LiLi4}. See also \cite{LLL24} for the isentropic case. 
When $p\neq 2$, it remains an interesting question whether we can prove the global well-posedness, regularity and the convergence of the system \eqref{NDEC3} on complete Riemannian manifolds with suitable geometric condition. We will study this problem in the future.
\end{remark}

\section{Lagrangian and Hamiltonian for the Langevin deformation }\label{LaHa}

In this section, we prove some variational formulas 
for the Lagrangian and Hamiltonian of the Langevin deformation \eqref{NDEC1} with $\mathcal{V}(\rho)=\mathrm{Ent}(\rho)$ and $\mathcal{V}(\rho)=-\mathrm{Ent}(\rho)$ respectively, which have their own interests. Then we prove the second main results of this paper, i.e., Theorem \ref{Ent-Riccati}.

\begin{theorem}\label{Lagrange}
 Let $(M, g)$ be a complete Riemannian manifold with bounded geometry condition,  $p>1$ and $q=\frac{p}{p-1}$. For any $c\ge0$,  let $(\rho(t),\phi(t)), t\in [0, T])$ be a smooth solution to \eqref{Lan}. Define the Lagrangian $L_c(\rho(t),\phi(t))$ as follows
\begin{equation*}
L_c(\rho(t),\phi(t)):=\frac{c^p}{q}\int_M|\nabla\phi(t)|^p\rho(t)\, dv-\int_M\rho(t)\log\rho(t)\, dv,\ \ \forall t\in [0, T].
\end{equation*}
Then, for all $t \in [0, T]$, we have
\begin{eqnarray}
\frac{d}{dt}L_c(\rho(t),\phi(t))&=&-p\int_M|\nabla\phi(t)|^{p-2}\langle\nabla\phi(t),\nabla\rho(t)\rangle\,dv-(p-1)\int_M|\nabla\phi(t)|^p\rho(t)\,dv,\label{dt_L}\\
\frac{d^2}{dt^2}L_c(\rho(t),\phi(t))&=&-p\int_M|\nabla \phi|^{2p-4}(|\nabla^2\phi|^2_A+{\rm Ric}(\nabla \phi,\nabla \phi))\rho\,dv \notag\\
&&+\frac p{c^p}\int_M|\nabla \phi|^{p-2}|\nabla\phi+\nabla\log\rho|^2_A\rho\,dv.\label{dt2_L}
\end{eqnarray}
Define the Hamiltonian $H_c(\rho(t),\phi(t))$ as follows
\begin{equation*}
H_c(\rho(t),\phi(t)):=\frac{c^p}{p}\int_M|\nabla\phi(t)|^p\rho(t)\, dv+\int_M\rho(t)\log\rho(t)\, dv,\ \ \forall t\in [0, T].  
\end{equation*}
Then, for all $t \in [0, T]$, we have
\begin{eqnarray}
\frac{d}{dt}H_c(\rho(t),\phi(t))&=&-\int_M|\nabla\phi(t)|^p\rho(t)\,dv\leq 0, \label{dt_H}\\
\frac{d^2}{dt^2}H_c(\rho(t),\phi(t))&=&-\frac{p}{c^p}\int_M\Delta_p\phi~\rho\,dv+\frac{p}{c^p}\int_M|\nabla\phi|^p\rho\,dv.\label{dt2_H}
\end{eqnarray}
\end{theorem}

\proof By \eqref{Lan}, a direct calculation  implies that
\begin{equation}\label{dtH}
\frac{d}{dt}\left(\int_M\rho\log\rho\,dv\right)=\int_M\partial_t\rho(1+\log\rho)\,dv
=\int_M|\nabla\phi|^{p-2}\langle\nabla\phi,\nabla\rho\rangle\,dv=-\int_M(\Delta_p\phi)\rho\,dv,
\end{equation}
and
\begin{align}\label{dtH2}
\frac{d}{dt}\left(\int_M|\nabla\phi|^p\rho\,dv\right)
=&-\frac{p}{c^p}\int_M|\nabla\phi|^{p-2}\langle\nabla\phi,\nabla\rho\rangle\,dv-\frac{p}{c^p}\int_M|\nabla\phi|^p\rho\,dv.
\end{align}
Combining \eqref{dtH} with \eqref{dtH2}, we derive \eqref{dt_L}, \eqref{dt_H} and \eqref{dt2_H}. 

Applying the $p$-Bochner formula \eqref{p-bochner} and \eqref{HJEnt2}, we have
\begin{align}\label{dtH3}
&\frac{d}{dt}\left(\int_M|\nabla\phi|^{p-2}\langle\nabla\phi,\nabla\rho\rangle\,dv\right)\notag\\
=&\int_M|\nabla \phi|^{2p-4}(|\nabla^2 \phi|^2_A+{\rm Ric}(\nabla \phi,\nabla \phi))\rho\,dv-\frac1{c^p}\int_M|\nabla\phi|^{p-2}\left(|\nabla\log\rho|^2_A+\langle\nabla\phi,\nabla\log\rho\rangle_A\right)\rho\,dv.
\end{align}
Combining \eqref{dtH2} with \eqref{dtH3}, we have \eqref{dt2_L}. 
\endproof

Similarly, we can prove the following variational formula.
\begin{theorem}\label{Hamilton}
Under the same condition as Theorem \ref{Lagrange}, let $(\rho(t),\phi(t)), t\in [0, T])$ be a smooth solution to the following equations 
 \begin{equation} \label{NDEC}
\left\{\begin{aligned}
    &\frac{\partial\rho}{\partial t}+\nabla\cdot\left(\rho|\nabla \phi|^{p-2}\nabla \phi\right)=0, \\
   &c^p\left(\frac{\partial \phi}{\partial t}+\frac1p|\nabla \phi|^p\right)=-\phi+\log\rho+1.
\end{aligned}
\right.
\end{equation}
Define the  Hamiltonian $H_c(\rho(t),\phi(t))$  as follows
\begin{equation*}
H_c(\rho(t),\phi(t)):=\frac{c^p}{q}\int_M|\nabla\phi(t)|^p\rho(t)\, dv+\int_M\rho(t)\log\rho(t)\, dv,\ \ \forall t\in [0, T].
\end{equation*}
Then, for all $t\in [0, T]$, we have
\begin{equation*}
\frac{d}{dt}H_c(\rho(t),\phi(t))=p\int_M|\nabla\phi|^{p-2}\langle\nabla\phi,\nabla\rho\rangle\,dv-(p-1)\int_M|\nabla\phi|^p\rho\,dv,
\end{equation*}
and
\begin{equation*}\label{dt2H}
\frac{d^2}{dt^2}H_c(\rho(t),\phi(t))=p\int_M|\nabla \phi|^{2p-4}\left(|\nabla^2\phi|^2_A+{\rm Ric}(\nabla \phi,\nabla \phi)+\frac 1{c^p}|\nabla\phi|^{2-p}|\nabla\phi-\nabla\log\rho|^2_A\right)\rho\,dv.
\end{equation*}
In particular, if  ${\rm Ric}\ge0$, then $H_c(\rho,\phi)$ is convex along the Langevin deformation of flows $(\rho,\phi)$ defined by \eqref{NDEC}.
\end{theorem}

\begin{corollary} Let $(M,g)$ be a complete Riemannian manifold with bounded geometry condition. Then
\begin{itemize}
\item When $c=0$, the Langevin deformation \eqref{Lan} reduces to the gradient flow of the Boltzmann entropy, i.e., a $p$-heat equation \eqref{BPHE}. By an analogous calculation in \eqref{dt2H}, we have
      $$
  \frac{d^2}{dt^2}{\rm Ent}(\rho,\phi)=p\int_M|\nabla \log\rho|^{2p-4}(|\nabla^2 \log\rho|^2_A+{\rm Ric}(\nabla \log\rho,\nabla \log\rho))\rho\,dv.
      $$
  \item When $c=\infty$, the Langevin deformation \eqref{Lan} becomes the $L^q$-geodesic flow on the $L^q$-Wasserstein space. In this case, by \eqref{HJEnt2}, we have
$$
  \frac{d^2}{dt^2}{\rm Ent}(\rho,\phi)=\int_M|\nabla \phi|^{2p-4}(|\nabla^2 \phi|^2_A+{\rm Ric}(\nabla \phi,\nabla \phi))\rho\,dv.
$$
  \end{itemize}
\end{corollary}

\proof[\bf Proof of Theorem \ref{Ent-Riccati}] By a direct calculation,  we have
\begin{equation}\label{dtE}
\frac{d}{dt}{\rm Ent}(\rho(t))=\int_M\partial_t\rho(1+\log\rho)\,dv
=\int_M|\nabla\phi|^{p-2}\langle\nabla\phi,\nabla\rho\rangle\,dv,
\end{equation}
and{}
\begin{equation}\label{dt2E}
\begin{aligned}
\frac{d^2}{dt^2}{\rm Ent}(\rho(t))=&\int_M|\nabla \phi|^{2p-4}(|\nabla^2 \phi|^2_A+{\rm Ric}(\nabla \phi,\nabla \phi))\rho\,dv-\frac1{c^p}\int_M|\nabla\phi|^{p-2}|\nabla\log\rho|^2_A\rho dv\\
&-\frac1{c^p}\int_M|\nabla\phi|^{p-2}\langle\nabla\phi,\nabla \rho\rangle_A\,dv.
\end{aligned}
\end{equation}
Thus, we obtain \eqref{dtH4}, that is
\begin{align*}
&\frac{d^2}{dt^2}{\rm Ent}(\rho)+\frac{p-1}{c^p}\frac{d}{dt}{\rm Ent}(\rho)+\frac1{c^p}\int_M|\nabla\phi|^{p-2}|\nabla\log\rho|^2_A\rho\,dv=\int_M|\nabla \phi|^{2p-4}\left(|\nabla^2 \phi|^2_A+{\rm Ric}(\nabla \phi,\nabla \phi)\right)\rho\,dv.
\end{align*}
 By the identity
$
{\rm tr}_A(|\nabla \phi|^{p-2}\nabla^2\phi)=\Delta_p\phi
$, we have
\begin{align}\label{dtH5}
\int_M\Big||\nabla \phi|^{p-2}\nabla^2\phi-\alpha(t)a \Big|^2_A\rho\,dv
=\int_M|\nabla \phi|^{2p-4}|\nabla^2\phi |^2_A\rho\,dv-2\alpha(t)\int_M\Delta_p\phi\rho\,dv+n\alpha(t)^2.
\end{align}
Substituting \eqref{dtH5} into \eqref{dtH4}, we have  
\begin{align}\label{dtH6}
&\frac{d^2}{dt^2}{\rm Ent}(\rho)+\left(2\alpha(t)+\frac{p-1}{c^p}\right)\frac{d}{dt}{\rm Ent}(\rho)+\frac1{c^p}\int_M|\nabla\phi|^{p-2}|\nabla\log\rho|^2_A\rho\,dv+n\alpha^2(t)\notag\\
=&\int_M\left[\Big||\nabla \phi|^{p-2}\nabla^2\phi-\alpha(t)a \Big|^2_A+|\nabla \phi|^{2p-4}{\rm Ric}(\nabla \phi,\nabla \phi)\right]\rho\,dv.
\end{align}
By the definition of $W_{c,n,p}$ \eqref{Wcn},  and \eqref{dtH6}, we obtain 
\begin{align}\label{dtwcn2}
\frac1{\eta(t)}\frac{d}{dt}W_{c,n,p}(\rho(t),t)=&\frac{d^2}{dt^2}{\rm Ent}(\rho(t))+\frac{1+\dot{\eta}(t)}{\eta(t)}\frac{d}{dt}{\rm Ent}(\rho(t))+n\left[\left(\frac{\dot{w}(t)}{w(t)}\right)'+\frac{1+\dot{\eta}(t)}{\eta(t)}\frac{\dot{w}(t)}{w(t)}\right]\notag\\
=&\frac{d^2}{dt^2}{\rm Ent}(\rho(t))+\left(2\alpha(t)+\frac{p-1}{c^p}\right)\frac{d}{dt}{\rm Ent}(\rho(t))+n\left[\dot{\alpha}(t)+2\alpha^2(t)+\frac{p-1}{c^p}\alpha(t)\right]\notag\\
=&\int_M\left[\Big||\nabla \phi|^{p-2}\nabla^2\phi-\alpha(t)a \Big|^2_A+|\nabla \phi|^{2p-4}{\rm Ric}(\nabla \phi,\nabla \phi)\right]\rho\,dv\notag\\
&-\frac1{c^p}\int_M|\nabla\phi|^{p-2}|\nabla\log\rho|^2_A\rho\,dv+n\left[\dot{\alpha}(t)+\alpha^2(t)+\frac{p-1}{c^p}\alpha(t)\right].
\end{align}
Combining \eqref{dtwcn2}, \eqref{alphaeq} with the definition of the relative Fisher information \eqref{Icn}, we obtain the $W$-entropy-information formula \eqref{wcnentropy}.
\endproof

Similarly to the proof of Theorem \ref{CEntForm} and Theorem \ref{Ent-Riccati}, we can extend the $W$-entropy-information formula \eqref{wcnentropy} to smooth solutions $(\rho, \phi)$ of the Langevin deformation equations with suitable growth conditions on complete Riemannian manifolds $M$ with bounded geometry condition. In view of this, we have the following rigidity theorem.
\begin{theorem}\label{RT_WIE}
Let $M$ be an $n$-dimensional complete Riemannian manifold with bounded geometry condition and with $\mathrm{Ric}\geq 0$. Suppose that $(\rho,\phi)$ is a smooth solution of the Langevin deformation \eqref{Lan} and satisfies suitable growth condition as in Proposition \ref{EDF}. Then the $W$-entropy-information quantity in \eqref{WEI} vanishes at some $t=t_0 >0$, i.e.,
\begin{equation*}
\frac1{\eta(t)}\frac{d}{dt}{W}_{c,n,p}(\rho(t),t)+\frac1{c^p}I_{c,n,p}(\rho(t),\phi(t))=0, 
\end{equation*}
at some $t = t_0 > 0$, if and only if $M$ is isometric to $\mathbb R^n$ and $(\rho,\phi)=(\rho_{c,n,p},\phi_{c,n,p})$ is the special solution to the Langevin deformation of flows on $TP_q(\mathbb R^n)$ as given in \eqref{RM for L flow1}.
\end{theorem}

\begin{proof}
The proof is similar to the one of Theorem \ref{CEntForm}. Indeed, we need only to use the entropy dissipation formulas \eqref{dtE} and \eqref{dt2E} to prove the $W$-entropy-information formula \eqref{wcnentropy} in Theorem \ref{Ent-Riccati}  for smooth solutions $(\rho, \phi)$ with suitable growth condition of the Langevin deformation on complete Riemannian manifold with bounded geometry condition. 

On the other hand, under the assumption ${\rm Ric}\ge 0$, if the $W$-entropy-information quantity in \eqref{WEI} vanishes at some $t=t_0 >0$, then the $W$-entropy-information formula \eqref{wcnentropy} yields
\begin{equation*}
|\nabla \phi|^{p-2}\nabla_i\nabla_j \phi=\alpha(t)a_{ij}, 
\end{equation*}
which is equivalent to 
\begin{equation*}
\nabla_i\nabla_j \phi=\frac{\alpha(t)}{|\nabla \phi|^{p-2}}\left(g_{ij}+(q-2)\frac{\nabla_i\phi\nabla_j\phi}{|\nabla \phi|^{2}}\right). 
\end{equation*}
By \eqref{wcnentropy} and using the same argument as in the proof of Theorem \ref{CEntForm}, we can prove Theorem \ref{RT_WIE}. 
\end{proof}

\medskip
    \noindent{\bf Declaration on Conflict of interest}. The authors state that there is no conflict of interest.

\noindent{\bf Acknowledgements}. 
We would like to express our gratitude to Dr. S. Li for valuable discussions during  the preparation of this work. The first named author would like to thank the Department of Mathematics of Tohoku University for support during the revision of this manuscript. 


Rong Lei, Academy of Mathematics and Systems Science, Chinese Academy of Sciences, No. 55, Zhongguancun East Road, Beijing, 100190, China

\emph{E-mail:} leirong17@mails.ucas.edu.cn

Xiang-Dong Li, State Key Laboratory of Mathematical Sciences, Academy of Mathematics and Systems Science, Chinese Academy of Sciences, No. 55, Zhongguancun East Road, Beijing, 100190, China,  
and 
 School of Mathematics, University of Chinese Academy of Sciences, Beijing, 100049, China.

\emph{E-mail:} xdli@amt.ac.cn

Yu-Zhao Wang, School of Mathematics and  Statistics, Shanxi University, Taiyuan, 030006, Shanxi, China\\ and Key Laboratory of Complex Systems and Data Science of Ministry of Education , Shanxi University, Taiyuan, 030006, Shanxi, China.

\emph{E-mail:} wangyuzhao@sxu.edu.cn


\begin{thebibliography}{00}
\bibitem{AGS}
L. Ambrosio, N. Gigli, G. Savar\'e. \emph{Gradient flows in metric spaces and in the space of probability measures}, Lectures in Mathematics, Birkhauer, 2005.

\bibitem{Aubin}T. Aubin. \emph{Nonlinear Analysis on Manifolds, Monge–Ampère Equations}. Springer, Berlin, 1980

%
%
\bibitem{BB}
J.-D. Benamou, Y. Brenier. \emph{A computational fluid mechanics solution to the Monge Kantorovich mass transfer problem}. Numer. Math. 84, (2000),  375–393.
%

\bibitem{Brasco}
L. Brasco.  \emph{A survey on dynamical transport distances}, J. Math. Sci. 181 (2012),  no.6, 755–781.
%
\bibitem{Bre}H. Brezis. \emph{Functional analysis, Sobolev spaces and partial differential equations}. Springer Science \& Business Media, 2010.
%



%
\bibitem{EKS}  M. Erbar, K. Kuwada, K.-T. Sturm. \emph{On the equivalence of the entropic curvature-dimension condition and Bochner's inequality on metric measure spaces}. Invent. Math. 201 (2015), 993–1071.

\bibitem{JKO} R. Jordan, D. Kinderlehrer, F. Otto. \emph{The variational formulation of the Fokker-Planck equation}. SIAM J. Math. Anal. 29 (1998), no. 1, 1–17.

\bibitem{Kant}
L.V. Kantorovich, \emph{On the translocation of masses}. C.R. (Dokl.) Acad. Sci. URSS(N.S.) 37 (1942), 199–201. 



\bibitem{KR}
L.V. Kantorovich, G. S. Rubinstein. \emph{On a space of totally additive functions}. Vestn. Leningrad. Univ. 13 (1958), no. 7, 52–59.


\bibitem{KoNi}
B. Kotschwar, L. Ni. \emph{Gradient estimate for $p$-harmonic functions, $1/H$ flow and an entropy formula}, Ann. Sci. \'Ec. Norm. Sup\'er., 42 (2009), no. 1, 1–36.


%
%
\bibitem{Kato1975} T. Kato. \emph{The Cauchy problem for quasi-linear symmetric hyperbolic systems}. Arch. Rational Mech. Anal. 58 (1975), 181–205.
%

%


\bibitem{LLL24}
R. Lei, S.Z. Li, X.-D. Li.  \emph{ Langevin deformation  for  R\'enyi entropy on Wasserstein space over Riemannian manifolds}, arXiv:2410.20369v1, 2024.

%
%
%







\bibitem{LiLi4}
S.  Li, X.-D. Li. \emph{ $W$-entropy formulas on super Ricci flows and Langevin deformation on Wasserstein space over Riemannian manifolds}, Sci. China Math.  61 (2018), no. 8, 1385–1406.

\bibitem{LiLi3}
S. Li, X.-D. Li. \emph{$W$-entropy and Langevin deformation on Wasserstein space over Riemannian  manifolds},  Probab. Theory Relat. Fields, 188 (2024), 911–955.



%


\bibitem{LiXD2}
X. -D. Li. \emph{Perelman's entropy formula for the Witten Laplacian on Riemannian manifolds via Bakry-Emery Ricci curvature}, Math. Ann. 353 (2012),  no. 2, 403–437.

%
%
%

\bibitem{Lott}
J. Lott. \emph{Optimal transport and Perelman's reduced volume}, Calc. Var. and Partial Differential Equations 36 (2009), 9–84.

\bibitem{LV}
J. Lott, C. Villani. \emph{ Ricci curvature for metric-measure spaces via optimal transport}, Annals of Math. 169 (2009), 903–991.
%
%
%
\bibitem{McC}R. J. McCann. \emph{Polar factorization of maps on Riemannian manifolds}. Geom. Funct. Anal.11 (3) (2001), 589–608.


\bibitem{Majda1984} A. Majda. \emph{Compressible fluid flow and systems of conservation laws in several space variables}, volume 53. Springer Science \& Business Media, 2012.

%

\bibitem{Ni}
L. Ni. \emph{The entropy formula for linear equation},  J. Geom. Anal. 14 (2004), no.1, 87–100.

%
%
\bibitem{Otto01}
F. Otto.  \emph{The geometry of dissipative evolution equations: the porous medium equation}, Commun. Partial Differ. Equ. 26 (2001), 101–174.

\bibitem{OV}F. Otto, C. Villani. \emph{Generalization of an inequality by Talagrandand links with the logarithmic Sobolev inequality}. J. Funct. Anal. 173 (2000), 361–400. 


\bibitem{Perelman}
G. Perelman. \emph{The entropy formula for the Ricci flow and its geometric applications}, arXiv.org/abs/maths0211159.

\bibitem{STW}T. Sideris, B. Thomases, D. H. Wang. \emph{Long time behavior of solutions to the 3D compressible Euler equations with damping}. Commun. PDE 28(3-4) (2003), 795–816.
 
\bibitem{RensSturm} M.-K.  von Renesse, K.-T. Sturm. \emph{Transport inequalities, gradient estimates, entropy, and Ricci curvature} Commun. Pure Appl. Math. 58 (7) (2005), 923–940.

 \bibitem{Sturm} K.-T. Sturm. \emph{On the geometry of metric measure spaces I}. Acta Math. 196(1) (2006),  
 65-131, and II. Acta Math.196(1) (2006), 133–177.

%
%
 
  \bibitem{V1}
 C. Villani. \emph{Topics in optimal transportation}. Graduate Studies in Mathematics, 58. American Mathematical Society, Providence, RI, 2003.

\bibitem{V2} C. Villani. \emph{ Optimal transport. Old and new}. Grundlehren der Mathematischen Wissenschaften, 338. Springer-Verlag, Berlin, 2009.

\bibitem{WY} W. Wang, T. Yang, \emph{The pointwise estimates of solutions for Euler equations with damping in multi-dimensions}. J. Differential Equations 173 (2001), no. 2, 410–450.



\end{thebibliography}
\end{document}